\documentclass[12pt]{article} 
\usepackage{amssymb,amsmath,latexsym,amsthm}
\input xy \xyoption{all}

\newtheorem{thm}{Theorem}[section] \newtheorem{lem}[thm]{Lemma}   

\theoremstyle{definition}

\theoremstyle{definition}
\newtheorem{remark}[thm]{Remark}

\xyoption{2cell}
\UseAllTwocells

\usepackage[T1]{fontenc}
\usepackage[all,2cell,poly,knot,tips]{xy}

\usepackage[all,2cell]{xy}
\usepackage{amsfonts}
\usepackage{amsthm}
\usepackage{amsmath}
\usepackage{amssymb}

\usepackage{color}
\usepackage[colorlinks]{hyperref}
\hypersetup{urlcolor=blue,linkcolor=red,citecolor=blue}
\usepackage{enumerate,xspace}

\usepackage{graphicx}
\graphicspath{ {images/} }

\DeclareMathOperator{\Span}{Span}
\DeclareMathOperator{\Cod}{Cod}
\DeclareMathOperator{\Dec}{Dec}

\usepackage{geometry} 
\usepackage{titlesec}
\usepackage{setspace}
\usepackage{array}

\begin{document}
\title{{\large{\textbf{SKEW MONOIDALES in SPAN}}}}
\author{Jim Andrianopoulos}
\date{\today}
\maketitle
\abstract{This article consists of an interesting characterisation of a skew monoidale in the monoidal bicategory $\Span$. After discussing the shift or decalage functor on simplicial sets we characterise these skew monoidales as categories $\mathbb C$ together with a functor $R \colon \Dec(\mathbb C) \longrightarrow \mathbb C$ which satisfies two conditions and give an example where the unit of the skew monoidale is not of a restricted type.}

\section{Introduction}
A general classification of skew monoidales in a monoidal bicategory in terms of simplicial maps from the Catalan simplicial set into the nerve of the monoidal bicategory is shown in \cite{CatSet}. We will consider a skew monoidale in the particular monoidal bicategory $\Span$  to examine a question arising from \cite{SkMon}.

Since their introduction by B{\'e}nabou in \cite{Bena}, $\Span$ and the $\Span$ construction are ubiquitous in higher category theory. This is mainly due to the fact that a category can be regarded as a monad in the bicategory of spans $\Span$, and various generalisations. However, what interests us is $\Span$ not just as a bicategory but as a monoidal bicategory made monoidal using the cartestian product of sets.  Skew monoidales (= skew pseudomonoids) were defined by Lack and Street in \cite{SkMon}, where they also show that quantum categories are skew monoidal objects, with a certain unit, in an appropriate monoidal bicategory. This contains as a special case the fact that categories are equivalently skew monoidales $C$ in the monoidal bicategory $\Span$ with tensor product given by
$$C\times C \stackrel{(s,t)}\longleftarrow E \stackrel{t}\longrightarrow C \ $$ for some set $E$,
and where the unit is assumed to be of the form
$$I \stackrel{!}\longleftarrow C \stackrel{1}\longrightarrow C \ ;$$ 
where $I$ is a terminal object in \textbf{Set}.
We characterise skew monoidales in $\Span$ without any restrictions on the unit of the skew monoidale. This means that the tensor product for the skew monoidale $C$ is given by
$$C\times C \stackrel{(s,r)}\longleftarrow E \stackrel{t}\longrightarrow C \ $$ for some set $E$, 
and  where the unit has the form
$$I \stackrel{!}\longleftarrow U \stackrel{j}\longrightarrow C \ .$$
This characterisation follows some lengthy but not difficult calculations, which are made easier using the concrete form a pullback takes in \textbf{Set}. We recover the fact in \cite{SkMon}, that categories are equivalently skew monoidales in $\Span$ with a unit of a certain restricted type. Section 5.2 collects the extra structure obtained from a skew monoidale in the form of a functor $R$ with some interesting properties.
We finish the article with a simple example of a skew monoidale (actually just a monoidale) in $\Span$ whose unit is not of the restricted type previously considered. 

\section{Skew Monoidales}
Instead of defining a skew monoidale in a general monoidal bicategory we define them in a Gray monoid. So in what follows we write as if $\cal B$ is a 2-category.   
 
Let $\cal B$ be a Gray monoid; see \cite{Day} for an explicit definition. Note that for 1-cells $f : A \longrightarrow A'$ and $g : B \longrightarrow B'$ in a Gray monoid, the only structural 2-cells are the invertible 2-cells of the form 
\begin{align*}
\xymatrix {
A\otimes B \ar[rr]^{1\otimes g} \ar[d]_{f \otimes 1} \ar@{}[drr]|{\cong} && A\otimes B' \ar[d]^{f\otimes 1}
\\ A'\otimes B \ar[rr]_{1\otimes g} && A'\otimes B'}
\end{align*} 
or tensors and composites thereof.
In this section we denote them with the symbol $\cong$ as above. These 2-cells satisfy some axioms which we do not list. We write $I$ for the unit object of the Gray monoid.

A {\em skew monoidal} structure on an object $A$ in $\cal B$ consists of morphisms $p \colon A\otimes A \longrightarrow A$ and $j \colon I \longrightarrow A$ in $\cal B$, respectively called the {\em tensor product} and {\em unit}, equipped with the following 2-cells, denoted by $\alpha$, $\lambda$ and $\rho$,respectively called the {\em associativity},{\em left unit} and {\em right unit constraints} 
\begin{equation*}\label{smonoidale1}
\begin{aligned}
\xymatrix{
 A\otimes A\otimes A \ar[rr]^{1\otimes p}_{~}="1" \ar[d]_{p\otimes 1}  && A\otimes A \ar[d]^p
  \\  A\otimes A \ar[rr]_p^{~}="2"   \ar@{=>}"2";"1"^\alpha && A }
\end{aligned}
 \end{equation*} 
$$\vbox{\xymatrix {
A\otimes A \ar[ddrr]_p && A \ar[ll]_{j\otimes 1} \ar[dd]^1_{~}="1"
\\
\\ && A
\ar@{=>}"1"+/va(-199)+3pc/;"1"_{\lambda}
}}
\quad   \quad
\vbox{\xymatrix {
&& A\ar[dd]^{1\otimes j}
\\
\\ A \ar[rruu]^1_{~}="1" && A\otimes A \ar[ll]^p
\ar@{<=}"1"-/va(-199)+3pc/;"1"_{\rho}
}}$$ 
subject to the following five axioms 
\begin{equation}\label{mpent}
\vcenter{
%\vbox {
\xymatrix@!=0.9pc{
&& A\otimes A\otimes A \ar[dd]_{1\otimes p}^(0.7){~}="2" \ar[drr]^{p\otimes 1}_(0.7){~}="1" 
\\ A\otimes A\otimes A\otimes A \ar[urr]^{p\otimes 1\otimes 1} \ar[dd]_{1\otimes 1\otimes p} \ar@{}[drr]|{\cong} &&&&
A\otimes A \ar[dd]^p%_(0.2){~}="1" 
\\ && A\otimes A\ar[drr]^p_{~}="3" 
\\ A\otimes A\otimes A \ar[drr]_{1\otimes p}^{~}="4" 
\ar[urr]^{p\otimes 1} &&&& A
\\ && A\otimes A \ar[urr]_p
\ar@{=>}"1";"2"_{\alpha}
\ar@{=>}"3";"4"_{\alpha}
  }
}
  = 
\vcenter{
%\vbox {
\xymatrix@!=0.9pc{
&& A\otimes A\otimes A \ar[drr]^{p\otimes 1}_{~}="3"
\\ A\otimes A\otimes A\otimes A \ar[urr]^{p\otimes 1\otimes 1} \ar[dd]_{1\otimes 1\otimes p}^{~}="4" 
\ar[drr]_{1\otimes p\otimes 1}^{~}="1" &&&&
A\otimes A \ar[dd]^p_{~}="5" 
\\ && A\otimes A\otimes A \ar[dd]^{1\otimes p}_{~}="2" \ar[urr]_{p\otimes 1}
\\ A\otimes A\otimes A \ar[drr]_{1\otimes p} &&&& A
\\ && A\otimes A \ar[urr]_p_{~}="6"
\ar@{=>}"3";"1"_{{\alpha}\otimes 1}
\ar@{=>}"2";"4"_{1\otimes {\alpha}}
\ar@{=>}"5";"6"+/va(135)+3.82pc/_{\alpha}}}
\end{equation}

\begin{align}
\vcenter{ \xymatrix {
A\otimes A \ar[d]_{1\otimes p} \ar[r]^{j\otimes 1\otimes1} \ar@{}[dr]|{\cong}
 & A\otimes A\otimes A \ar[d]_{1\otimes p}^{~}="2" \ar[r]^{p\otimes 1}  & A\otimes A \ar[d]^p_{~}="1"
\\ A \ar[r]^{j\otimes 1} \ar@/_2pc/[rr]_1^{~}="3" & A\otimes A \ar[r]^p & A
\ar@{=>}"1";"2"_{\alpha}
\ar@{<=}"3";"3"+/va(90)+1.1pc/^{\lambda} 
} }
& \quad=\quad
\vcenter{ \xymatrix {
A\otimes A \ar[d]_{1\otimes p} \ar[r]^{j\otimes 1\otimes1} \ar@/_1.7pc/[rr]_1^{~}="1"
 & A\otimes A\otimes A  \ar[r]^{p\otimes 1}  & A\otimes A \ar[d]^p
\\ A \ar@/_2pc/[rr]_1 \ar@{}[rr]^{\cong}  && A
\ar@{<=}"1";"1"+/va(90)+1.1pc/^{\lambda\otimes 1}
} } \label{mleft} \\
\vcenter{
\xymatrix {
A\otimes A \ar[d]_{p\otimes 1} \ar[r]^{1\otimes 1\otimes j} \ar@{}[dr]|{\cong}
 & A\otimes A\otimes A \ar[d]_{p\otimes 1} \ar[r]^{1\otimes p}_{~}="2"  & A\otimes A \ar[d]^p
\\ A \ar[r]_{1\otimes j} \ar@/_2pc/[rr]_1^{~}="3" & A\otimes A \ar[r]_p^{~}="1" & A
\ar@{=>}"1";"2"_{\alpha}
\ar@{=>}"3";"3"+/va(90)+1.1pc/^{\rho} 
} }
& \quad=\quad 
%\vbox{
\vcenter{
\xymatrix {
A\otimes A \ar[d]_{p\otimes 1} \ar[r]^{1\otimes 1\otimes j} \ar@/_1.7pc/[rr]_1^{~}="1"
 & A\otimes A\otimes A  \ar[r]^{1\otimes p}  & A\otimes A \ar[d]^p
\\ A \ar@/_2pc/[rr]_1 \ar@{}[rr]^{\cong} && A
\ar@{=>}"1";"1"+/va(90)+1.1pc/^{1\otimes {\rho}}
} }
\label{mright} \\
\vcenter{
  \xymatrix {
A\otimes A \ar[rr]^1_{~}="1" \ar[dr]_{1\otimes j\otimes 1} \ar@/_1.6pc/[ddr]_1^{~}="2" && A\otimes A \ar[dd]^p_{~}="3"
\\ & A\otimes A\otimes A \ar[d]_{1\otimes p}^{~}="4" \ar[ru]_{p\otimes 1}
\\ & A\otimes A \ar[r]_p & A
\ar@{=>}"1";"1"+/va(-90)+1.3pc/^{{\rho}\otimes 1}
\ar@{<=}"2";"2"+/va(-315)+1.3pc/_{1\otimes {\lambda}}
\ar@{=>}"3";"4"^{\alpha}
} } 
&\quad=\quad 
\vcenter{
\xymatrix {
A\otimes A \ar[rr]^1 \ar[dd]_1 \ar@{}[ddrr]|{\equiv} && A\otimes A \ar[dd]^p
\\
\\ A\otimes A \ar[rr]_p && A
} }
\label{mmid} \\
\vcenter{
  \xymatrix {
I \ar[rr]^j \ar[d]_j \ar@{}[drr]|{\cong} && C \ar[d]^{1\otimes j} \ar@/^0.9pc/[drr]^1_{~}="1"
\\ C \ar[rr]_{j\otimes 1} \ar@/_1.7pc/[rrrr]_1_{~}="2" && C\otimes C \ar[rr]_p && C
\ar@{=>}"1";"1"+/va(-135)+1.3pc/^{\rho}
\ar@{<=}"2";"2"+/va(90)+1.3pc/^{\lambda}
} } 
&\quad=\quad 
\vcenter{
\xymatrix {
I \ar[rr]^j \ar[d]_j \ar@{}[drr]|{\equiv} && C \ar[d]^1
\\ C \ar[rr]_1 && C
} } \label{muu}
\end{align} 

An object $A$ of $\cal B$ equipped with such a skew monoidal structure is called a {\em skew monoidale} in $\cal B$;

A skew monoidale in the cartesian monoidal 2-category \textbf{Cat} of categories, functors and natural transformations is a skew monoidal category.

\section{Span as a Monoidal Bicategory}
In this chapter we are interested in the case where $\cal B$ is $\Span$. We first remind the reader of some details of $\Span$.

The objects of $\Span$ are those of \textbf{Set} ; so $A$, $B$, $C$ ..... are sets.
We denote the terminal object in \textbf{Set} by $1$ and the unique arrow into it by $!$.

An arrow $\xymatrix{r \colon A \ar[r]|(.56){\vert} & B}$ is a span $r = (f , R , g)$ in \textbf{Set}, as in $(a)$, where composition of these arrows is by pullback (pullback along $g$ and $h$), as in $(b)$, and the identity arrow is the span $(c)$ below.
$$
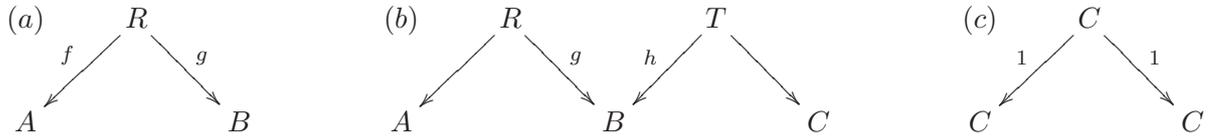
\vbox{\xymatrix{
(a) & R \ar[dl]_f \ar[dr]^g
\\ A && B }}
\qquad  \qquad
\vbox{\xymatrix {(b)
& R \ar[dl] \ar[dr]^g && T \ar[dl]_h \ar[dr] 
\\ A && B && C }}
\qquad  \qquad
\vbox{\xymatrix{(c)
& C \ar[dl]_1 \ar[dr]^1
\\ C && C }}$$
A 2-cell from $(w,R,x)$ to $(f,S,g)$ is a map $\tau \colon  R \longrightarrow S $ in \textbf{Set} such that the following commutes  
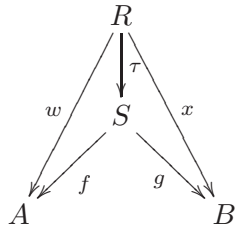
\begin{align*} 
\xymatrix {
& R \ar[ddl]_w \ar[ddr]^x \ar[d]^{\tau}
\\ & S \ar[dr]_g \ar[dl]^f
\\ A && B}
\end{align*}
As \textbf{Set} is a category with finite products as well as pullbacks (in the presence of a terminal object, finite products can be obtained as a special case of pullbacks) then the bicategory $\Span$ has a  monoidal product on it induced by the cartesian product of sets.
 
To calculate a left whiskering, such as in the following  diagram (on the left), we first form the pullbacks of $f$ and $w$ along $v$, then we use the fact that $f\tau = w$ and $v1 = v$ to construct the dotted arrow in the diagram on the right.
$$
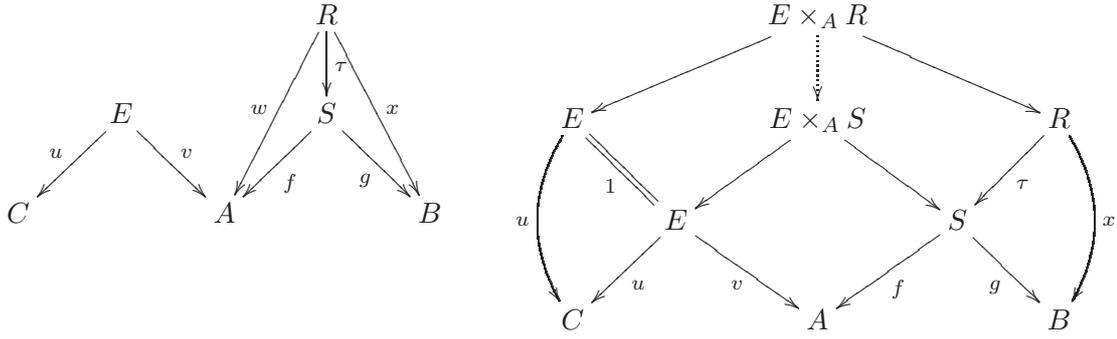
\vbox{\xymatrix {
&&& R \ar[ddl]_w \ar[ddr]^x \ar[d]^{\tau}
\\ & E \ar[dr]^v \ar[dl]_u && S \ar[dr]_g \ar[dl]^f
\\ C && A && B }}
\qquad 
\vbox{\xymatrix { && E\times_A R \ar[drr] \ar[dll] \ar@{..>}[d]
\\ E  \ar@/_1.1pc/[dd]_u \ar@{=}[dr]_1 && E\times_A S \ar[dr] \ar[dl] && R  \ar@/^1.1pc/[dd]^x \ar[dl]^{\tau}
\\ & E \ar[dr]_v \ar[dl]^u && S \ar[dr]_g \ar[dl]^f
\\ C  && A && B}}$$

\section{Notation and Calculations}

The motivation for this section is from \cite{SkMon} where skew monoidales in $\Span$ with a unit of the form $\xymatrix{(!,C,1) \colon 1 \ar[r]|(.65){\vert} & C}$ are shown to be equivalent to categories. Here we give a characterisation of a general skew monoidale in $\Span$.

Consider a skew monoidale in $\Span$ with underlying object the set $C$.

\textbf{The 1-cells of a Skew Monoidale:}

The tensor $\xymatrix{ p \colon  C\times C \ar[r]|(.65){\vert} & C}$ has the form
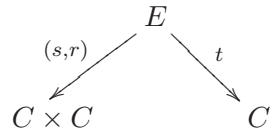
\begin{align*}
\xymatrix{
& E \ar[dl]_{(s,r)} \ar[dr]^t
\\ **[r]C\times C && C}
\end{align*}
So for $f \in E$ we will record this data as $\xymatrix{ {s(f)} \ar@{-->}[r]^f & {t(f)}}$ and $r(f) \in C$.

The unit $\xymatrix{ j \colon  1 \ar[r]|(.53){\vert} & C}$ has the form 
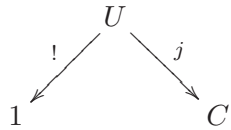
\begin{align*}
\xymatrix{
& U \ar[dl]_{!} \ar[dr]^j
\\ 1 && C}
\end{align*}
So for $u \in U$ we will record this data as $j(u) \in C$.

Given a skew monoidale, with its unit having the restricted form $\xymatrix{(!,C,1) \colon 1 \ar[r]|(.65){\vert} & C}$, it will become evident when dealing with the general case below, that this forces the first span to be of the form 
$C\times C \stackrel{(s,t)}\longleftarrow E \stackrel{t}\longrightarrow C \ $, and that it defines a category with $E$ as its set of arrows. Conversely, given a category $C = (C_1,C_0,1,s,t,\circ)$, we construct the following two spans: $C_0\times C_0 \stackrel{(s,t)}\longleftarrow C_1 \stackrel{t}\longrightarrow C_0 \ $, and $1 \stackrel{!}\longleftarrow C_0 \stackrel{1}\longrightarrow C_0 \ $. The 2-cell structure making this category into a skew monoidale comes from the composition and identity arrows of the category, with the skewness arising from the non-symmetric nature of the first span.

\textbf{The 2-cells of a Skew Monoidale:}

What is now required is a long and often repetitive calculation with, when we include the equations between the 2-cells, sixteen pullback constructions in \textbf{Set}; so we will present enough of it to introduce and justify the supporting notation that will form our input for a further characterisation.

\textbf{For the 2-cell} $\lambda \colon  p(j \times 1) \Longrightarrow 1$
we need to consider the following composite
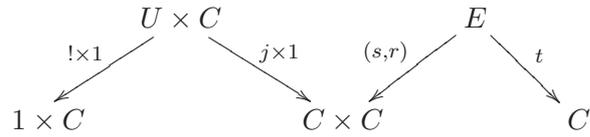
\begin{align*}
\xymatrix{
& U \times C \ar[dl]_{! \times 1} \ar[dr]^{j \times 1} && E \ar[dl]_{(s,r)} \ar[dr]^t
\\ **[r]1 \times C &&  C \times C && C}
\end{align*}  
First we need to form the following pullback
\begin{equation}\label{pull}
\xymatrix{ P \ar[r]^q \ar[d]_p & E \ar[d]^s
\\ U \ar[r]_j & C}
\end{equation}
then the required composite is
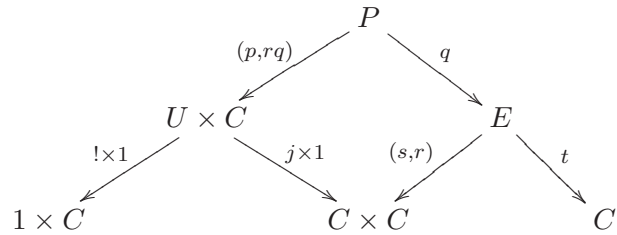
\begin{align*} 
\xymatrix{ && P \ar[dl]_{(p,rq)} \ar[dr]^q
\\ & U \times C \ar[dl]_{! \times 1} \ar[dr]^{j \times 1} && E \ar[dl]_{(s,r)} \ar[dr]^t
\\ 1 \times C &&  C \times C && C}
\end{align*} 
so we finally have for the 2-cell $\lambda$, a function which we also denote by $\lambda$, such that the following diagram commutes,
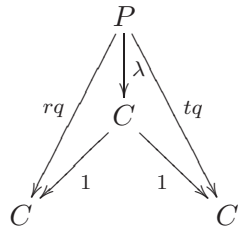
\begin{align*}
\xymatrix {& P \ar[ddl]_{rq} \ar[ddr]^{tq} \ar[d]^{\lambda}
\\ & C \ar[dr]_1 \ar[dl]^1
\\ C && C}
\end{align*}
it can only exist if $rq = tq$ and is then given as a morphism in \textbf{Set} by the common value
\begin{equation}\label{Dom}
\begin{aligned}
rq = tq.
\end{aligned}
\end{equation}
As we are in \textbf{Set} we can write $P$ as $P = \big\{(u,f)|  u \in U , f \in E, \hspace{.2cm} j(u) = s(f) \big\}$ with $p(u,f) = u$ and $q(u,f) = f$ as the projections.
With our notation, the elements in $P$ look like 
$\xymatrix { j(u) \ar@{-->}[r]^f & y}$.
We can now record the effect of $\lambda$ as :
$\xymatrix { j(u) \ar@{-->}[r]^f & y  \ar@{|->}[rr]^(0.40){\lambda} && y = r(f). }$
Thus the existence of $\lambda$ implies that if 
$\xymatrix { j(u) \ar@{-->}[r]^f & y}$ then $y = r(f)$, and the map itself sends $(u,\xymatrix { j(u) \ar@{-->}[r]^f & y})$ to $y$.

In the case of a category (that is, the case where $U$ is $C$ and the unit is of the form  $1 \stackrel{!}\longleftarrow C \stackrel{1}\longrightarrow C \ $) then $P = E = C_1$ and $j = 1$ forces $r = t$, so $\lambda$ is just $t$ .

\textbf{For the 2-cell} $\rho \colon  1 \Longrightarrow p(1 \times j)$ 
we first need to construct the following pullback
\begin{align*}
\xymatrix{ & B \ar[d]_m \ar[r]^k & E \ar[d]^r \\ & U \ar[r]_j & C }
\end{align*}
In the diagram below 
\begin{equation}\label{Rho}
\xymatrix{ && C \ar[d]^{\rho} \ar@/_1.8pc/[dddll]_1 \ar@/^1.2pc/[ddr]_{\phi} \ar@/_1.2pc/[ddl]^{(1,\psi)}
\ar@/^1.8pc/[dddrr]^1
\\ && B \ar[dr]_k \ar[dl]^{(sk,m)}
\\ & C \times U \ar[dr]_{1 \times j} \ar[dl]^{proj_1} && E \ar[dr]_t \ar[dl]^{(s,r)}
\\ C  && C \times C  && C}
\end{equation}
the square is the pullback involved in the composite $p(1\times j)$, so to give $\rho \colon  1 \Longrightarrow p(1 \times j)$ is equivalently to give $\phi \colon C \longrightarrow E$ and $\psi \colon C \longrightarrow U$ satisfying $t\phi = 1$, $s\phi = 1$, and $r\phi = j\psi$. 
We record for later use that
\begin{equation*}
\begin{aligned}
r\phi = j\psi.
\end{aligned}
\end{equation*}
As we are in \textbf{Set}, $B = \big\{(u,f)|  u \in U , f \in E, \hspace{.2cm}  j(u) = r(f) \big\}$ with $m(u,f) = u$ and $k(u,f) = f$ as the projections.
With respect to our notation, the elements in $B$ look like 
$ (j(u) = r(f) , \xymatrix{ x \ar@{-->}[r]^f & y} )$
so we record the effect of $\phi$ as
\begin{equation*} 
\xymatrix {x\in C  \ar@{|->}[r]^(.70){\phi} & } ( \xymatrix{ x \ar@{-->}[r]^{\phi_x} & x} )
\end{equation*} 
then $\psi_{x} \in U$ satisfies $j({\psi_x}) =  r(\phi_x)$.

Note that in the case of a category then $B = E = C_1$ and so $\rho$ is just the identity.

\textbf{For the 2-cell} $\alpha \colon  p(p \times 1) \Longrightarrow p(1 \times p)$
we need the following two pullbacks  
$$ \vbox{ \xymatrix{ & X \ar[d]_h \ar[r]^l & E \ar[d]^s \\ & E \ar[r]_t & C}}
\qquad  \qquad
\vbox{ \xymatrix{ & Y \ar[d]_e \ar[r]^y & E \ar[d]^r \\ & E \ar[r]_t & C}} $$
The objects $X$ and $Y$ will appear as the vertex of the spans $p(p\times 1)$ and $p(1\times p)$, respectively.

In the diagram below 
\begin{equation}\label{Alp}
\begin{aligned}
\xymatrix{ && X \ar[d]^{\alpha} \ar@/_1.8pc/[dddll]_{(sh,rh,rl)} \ar@/^1.2pc/[ddr]_{\delta} \ar@/_1.2pc/[ddl]^{(sh,\tau)}
\ar@/^1.8pc/[dddrr]^{tl}
\\ && Y \ar[dr]_y \ar[dl]^{(sy,e)}
\\ & C \times E \ar[dr]_{1 \times t} \ar[dl]^{1\times (s,r)} && E \ar[dr]_t \ar[dl]^{(s,r)}
\\ C\times C\times C  && C\times C  && C}
\end{aligned}
\end{equation}
the square is the pullback involved in the composite $p(1\times p)$, so to give $\alpha \colon  p(p\times 1) \Longrightarrow p(1 \times p)$ is equivalently to give $\tau \colon X \longrightarrow E$ and $\delta \colon X \longrightarrow E$ satisfying $t\delta = tl$, $s\delta = sh$,$s\tau = rh$, $r\tau = rl$, and $r\delta = t\tau$. 
 
As we are in \textbf{Set},  $X = \big\{(f,g)|  f , g \in E ;  t(f) = s(g) \big\}$ with $l(f,g) = g$ and $h(f,g) = f$ as the projections. Similarly, $Y = \big\{(f,g)|  f , g \in E ;  t(f) = r(g) \big\}$ with $y(f,g) = g$ and $e(f,g) = f$ as its projections. So with respect to our notation, elements of $X$ look like
$\xymatrix {x \ar@{-->}[r]^f & y \ar@{-->}[r]^g & z}$ and elements of $Y$ look like ($\xymatrix {x \ar@{-->}[r]^f & r(g)}$ , $\xymatrix{ y \ar@{-->}[r]^g & z}$) with $r(f)$, $r(g)$ $\in C$
and we record the effect of $\delta$ as
\begin{align*} 
\xymatrix {x \ar@{-->}[r]^f & y \ar@{-->}[r]^g & z  \ar@{|->}[rr]^{\delta} && } \xymatrix{ x \ar@{-->}[rr]^{gf} && z} 
\end{align*} 
and $\tau$ as
\begin{align*} 
\xymatrix {x \ar@{-->}[r]^f & y \ar@{-->}[r]^g & z  \ar@{|->}[rr]^{\tau} && } \xymatrix{ r(f) \ar@{-->}[rr]^{g^f} && r(gf) } 
\end{align*}

 with
 $r(g^f)$ = $r(g)$ in $C$.
 
\textbf{Note} that $\delta$ gives us a map from $x$ to $z$ which we have called $gf$. We want to interpret the set $E$ as a set of arrows and $gf$ as a composite (with $\phi_x$ as an identity), indeed, that this is the composite in a category will be shown below. The map $\tau$ gives us a map from $r(f)$ to $r(gf)$ which we have called $g^f$. This map will form the basis of our characterisation for the resulting ``extra'' structure given on the category. 

We now consider the \textbf{equations between the 2-cells} and just do one calculation to give the reader an indication of how the final relations are obtained.  
Consider the left-hand side of equation \eqref{mpent} and the whiskering
\begin{align*}  
\xymatrix {C \times C\times C \times C \ar[rr]^{p \times 1\times 1} && C \times C\times C \ar[rr]^{p \times 1} \ar[d]_{1 \times p}^{~}="2" && C \times C \ar[d]^p_{~}="1" 
\\ && C \times C \ar[rr]_p && C
\ar@{=>}"1";"2"_{\alpha}}
\end{align*} 
For this we need to compose
\begin{align*} 
\xymatrix {&&& X \ar@/_1.1pc/[ddl]_{(sh,rh,rl)} \ar@/^1.1pc/[ddrr]^{tl} \ar[d]|-{(\tau ,\delta)}
\\ & E \times C\times C \ar[dr]^{t \times 1\times 1} \ar[dl]_{(s,r) \times 1\times 1} && Y \ar[drr]_{ty} \ar[dl]^{(sy,se,re)}
\\ C \times C\times C\times C && C \times C \times C &&& C}
\end{align*}
First form the pullbacks
$$\vbox{ \xymatrix{ & Q \ar[d]_{y'} \ar[r]^{l'} & Y \ar[d]^y \\ & X \ar[r]_l & E}}
\qquad  \qquad 
\vbox{ \xymatrix{ & W \ar[d]_{h'} \ar[r]^w & X \ar[d]^h \\ & X \ar[r]_l & E}} $$
then we have the following pullbacks 
$$\vbox{ \xymatrix{ Q \ar[rr]^{l'} \ar[d]_{(hy',sel',rel')} && Y \ar[d]^{(sy,se,re)}
\\ E \times C\times C \ar[rr]_{t \times 1\times 1} && C \times C\times C}}  
\qquad  \qquad 
\vbox{ \xymatrix{W \ar[rr]^w \ar[d]_{(hh',rhw,rlw)}  && X \ar[d]^{(sh,rh,rl)}
\\ E \times C\times C \ar[rr]_{t \times 1\times 1} && C \times C\times C}} $$ 
and so form
\begin{align*} 
\xymatrix { && W \ar[drrr]^w \ar[dll]_{(hh',rhw,rlw)} \ar[d]^{\mu}
\\ E \times C\times C \ar@/_1.1pc/[dd]_{(s,r) \times 1\times 1} \ar@{=}[dr]_1 && Q \ar[dr]^{l'}	 \ar[dl]_{(hy',sel',rel')} &&& X  \ar@/^1.1pc/[dd]^{tl} \ar@ {=>}[dll]|-{(\tau ,\delta)}
\\ & E \times C\times C \ar[dr]_{t \times 1\times 1} \ar[dl]^{(s,r) \times 1\times 1} && Y \ar[drr]_{ty} \ar[dl]^{(sy,se,re)}
\\ C \times C\times C\times C && C \times C \times C &&& C}
\end{align*} 

As before, to give the map $\xymatrix{p(p\times 1)(p\times 1\times 1) \ar[rr]^{\alpha(p\times 1\times 1)} && p(1\times p)(p\times 1\times 1)}$ is equivalently to give $\gamma \colon W \longrightarrow E$ and $\epsilon \colon W \longrightarrow Y$ as in the diagram below. 
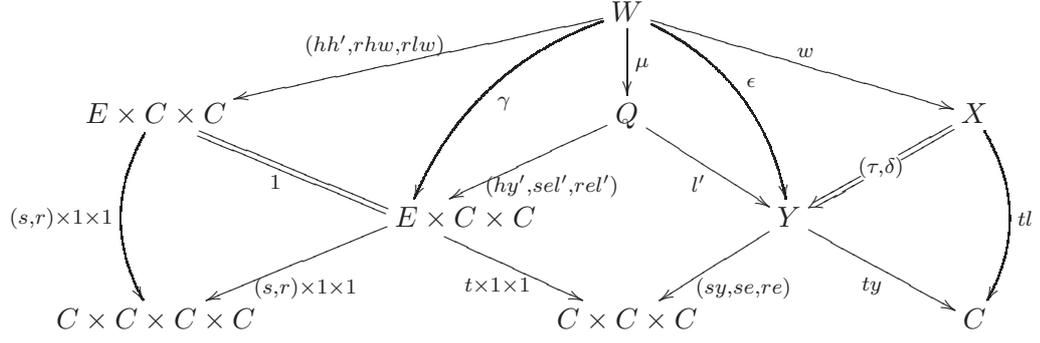
\begin{align*}  
\xymatrix { && W \ar[drrr]^w \ar[dll]_{(hh',rhw,rlw)} \ar[d]^{\mu} \ar@/^1.2pc/[ddr]^{\epsilon} \ar@/_1.2pc/[ddl]^{\gamma}
\\ E \times C\times C \ar@/_1.1pc/[dd]_{(s,r) \times 1\times 1} \ar@{=}[dr]_1 && Q \ar[dr]_{l'} \ar[dl]^{(hy',sel',rel')} &&& X  \ar@/^1.1pc/[dd]^{tl} \ar@ {=>}[dll]|-{(\tau ,\delta)}
\\ & **[r]E \times C\times C \ar[dr]_{t \times 1\times 1} \ar[dl]^{(s,r) \times 1\times 1} && Y \ar[drr]_{ty} \ar[dl]^{(sy,se,re)}
\\ C \times C\times C\times C && C \times C \times C &&& C}
\end{align*}
From this diagram we now establish some relationship between $(\gamma,\epsilon)$ and $(\tau,\delta)$. We get $\gamma = hh'$ and $\epsilon = (\tau,\delta)w = ({\tau}w,{\delta}w)$.
Now writing these as functions into just the set $E$ we recall the previous pullbacks we had constructed and consider the following diagram
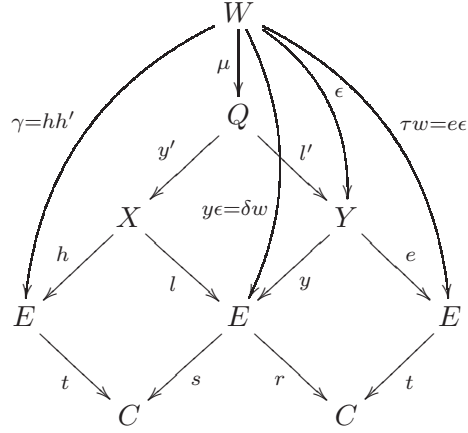
\begin{align*} 
\xymatrix{ && W \ar[d]_{\mu} \ar@/_1.8pc/[dddll]_{\gamma = hh'} \ar@/^1.2pc/[ddr]^{\epsilon}
\ar@/^1.8pc/[dddrr]^{{\tau}w = e{\epsilon}}
\ar@/^1.3pc/[ddd]_(.65){y{\epsilon} ={\delta}w}
\\ && Q \ar[dr]^{l'} \ar[dl]_{y'}
\\ & X \ar[dr]_l \ar[dl]_h && Y \ar[dr]^e \ar[dl]^y
\\ E \ar[dr]_t && E \ar[dl]^s \ar[dr]_r && E \ar[dl]^t
\\ & C && C}
\end{align*} 
From this diagram we have $\gamma = hh'$ , $y \epsilon = y(\tau w,\delta w ) = \delta w$ and $e \epsilon = e(\tau w,\delta w ) = \tau w$.

As we are in \textbf{Set}, $W = \big\{(x_1,x_2)|  x_1 , x_2 \in X ;  l(x_1) = h(x_2) \big\}$ with projections $h'(x_1,x_2) = x_1$ and $w(x_1,x_2) = x_2$. Similarly, $Q = \big\{(x,z)|  x \in X , z \in Y ;  l(x) = y(z) \big\}$ with projections $y'(x,z) = x$ and $l'(x,z) = z$.
So with respect to our notation, the elements of the set $W$ look like
$\xymatrix {a \ar@{-->}[r]^f & b \ar@{-->}[r]^g & c \ar@{-->}[r]^k & d }$ with  $r(f)$, $r(g)$, $r(k)$ $\in C$ and  the elements of $Q$ look like
 ($\xymatrix {a \ar@{-->}[r]^f & b \ar@{-->}[r]^g & c}$ , $\xymatrix{ v \ar@{-->}[r]^k & r(g)}$) with $r(f)$, $r(g)$ and $r(k)$ $\in C$.
So the 2-cell under consideration gives for $W \longrightarrow Q$ that 

$\xymatrix {a \ar@{-->}[r]^f & b \ar@{-->}[r]^g & c \ar@{-->}[r]^k & d }$  $\xymatrix{  \ar@{=>}[rr] && }$ $(\xymatrix {a \ar@{-->}[r]^f & b}$, $\xymatrix{ r(g) \ar@{-->}[r]^{k^g} & r(kg)}$, $\xymatrix{ b \ar@{-->}[r]^{kg} & d}$) 

Now observe that the two sides of the cube \eqref{mpent} act as in the diagram below,
\begin{equation*}
\begin{aligned}
\xymatrix{
& (gf,g^f,h) \ar@{|->}[rr] ^-{\phantom{A}\alpha (1\otimes p\otimes 1)}  &&  (gf,h^{gf},h(gf)) \ar@{|->}[d]^-{p(1\otimes \alpha)} \\
(f,g,h) \ar@{|->}[ru]^-{p(\alpha \otimes 1)\phantom{A}} \ar@{|->}[d]_-{\alpha(p\otimes 1 \otimes 1)} & & &  ((h^{gf})^{g^f},h^{gf}g^f,h(gf))  \\
(f,h^g,hg) \ar@{|->}[rd]_-{\cong} & & & (h^g,(hg)^f,(hg)f) &   \\
& (h^g,f,hg) \ar@{|->}[rur]_(.5){\alpha(1\otimes 1 \otimes p )  }} 
\end{aligned}
\end{equation*}
so the cube commutes if and only if the two expressions in the lower right corner agree; in other words, if the following equations, $h^g = (h^{gf})^{g^f}$, $(hg)^f = h^{gf}g^f$, and $(hg)f = h(gf)$ hold, for a composable triple of arrows. The remaining four equations are analyzed similarly, and the results summarized below.

\textbf{Summary:}
We summarize  all the calculations with respect to a skew monoidale into the notation introduced earlier to get:

\underline{For the 1-cell $\xymatrix{ p \colon  C\times C \ar[r]|(.65){\vert} & C}$ with vertex $E$} : For $f \in E$ , $\xymatrix{x \ar@{-->}[r]^f & y}$ for $x , y \in C$  and $r(f) \in C$.

\underline{For the 1-cell $\xymatrix{ j \colon  1 \ar[r]|(.53){\vert} & C}$ with vertex $U$} : For $u \in U$ that $j(u) \in C$.

\underline{For the 2-cell $\lambda$} : if $\xymatrix { j(u) \ar@{-->}[r]^f & y}$ then $y = r(f)$ in $C$.

\underline{For the 2-cell $\rho$} : for $x\in C$ we have $\xymatrix{ x \ar@{-->}[r]^{\phi_x} & x}$ in $E$ and $\psi_x \in U$ with $j({\psi_x}) =  r(\phi_x)$. 

\underline{For the 2-cell $\alpha$} :
if $\xymatrix {x \ar@{-->}[r]^f & y \ar@{-->}[r]^g & z }$ then  $\xymatrix{ r(f) \ar@{-->}[r]^{g^f} & r(gf) }$ and $\xymatrix{ x \ar@{-->}[r]^{gf} & z}$
are both in $E$ 

with $r(g^f)$ = $r(g)$.

\underline{For the equation between the 2-cells involving $(\lambda,\rho)$} : For $j(u) \in C$ we have 
${\psi}_{j(u)} = u$, 

that is, $\psi j = 1$.

\underline{For the equation between the 2-cells involving $(\rho,\alpha)$} : For  $\xymatrix{x \ar@{-->}[r]^f & y}$ we have

${\psi}_y = {\psi}_{r(f)}$, ${\phi}_y f = f$, and $ {\phi}^f_y = {\phi}_{r(g)}$. 

\underline{For the equation between the 2-cells involving $(\lambda,\alpha)$} : For $\xymatrix { j(u) \ar@{-->}[r]^f & y \ar@{-->}[r]^g & z }$ we

have $g^f = g$.

\underline{For the equation between the 2-cells involving $(\rho,\alpha,\lambda)$} : For  $\xymatrix{x \ar@{-->}[r]^f & y}$ we

have $f{\phi}_x = f$. 

\underline{For the equation between the 2-cells involving $(\alpha,\alpha)$ (the pentagon)} : 

For $\xymatrix {x \ar@{-->}[r]^f & y \ar@{-->}[r]^g & z \ar@{-->}[r]^h & a }$ we have
$(hg)f = h(gf)$ , $(hg)^f = h^{gf}g^f$, and
 $h^g = (h^{gf})^{g^f}$.

\vspace{.8cm}
 
We conclude that we can now safely rename $\phi_x$ as $1_x$ and change our notation for $\xymatrix{x \ar@{-->}[r]^f & y}$ to an arrow $\xymatrix{x \ar[r]^f & y}$ and with the condition that $(hg)f = h(gf)$ obtain a category with some extra structure consisting of :

\hspace{2cm}(a) for each morphism $f$ an object $r(f)$.

\hspace{2cm}(b) a set $U$ with a function $j$ from $U$ to the set of objects. 

\hspace{2cm}(c) for each composable pair $\xymatrix {x \ar[r]^f & y \ar[r]^g & z}$ a map  $\xymatrix{ r(f) \ar[r]^{g^f} & r(gf)}$ with $r(g^f) = r(g)$.

\hspace{2cm}(d) for each object $c$ an element $\psi_{c} \in U$.

\vspace{.4cm}

satisfying the following
\begin{eqnarray}
\textnormal{For} \xymatrix { u \in U} & \textnormal{that} & \psi j(u) = u. \label{11} \\
\textnormal{For} \xymatrix{x \in C} & \textnormal{that} & r(1_x) = j\psi. \label{12} \\
\textnormal{For} \xymatrix { j(u) \ar[r]^f & y \ar[r]^g & z } & \textnormal{that} & g^f = g. \label{13} \\
\textnormal{For} \xymatrix{x \ar[r]^f & y},r(f) \in C & \textnormal{that} & 1^f_y = 1_{r(f)}. \label{14} \\
\textnormal{For} \xymatrix{x \ar[r]^f & y},r(f) \in C & \textnormal{that} & \psi_y = \psi_{r(f)}. \label{15} \\
\textnormal{For} \xymatrix {x \ar[r]^f & y \ar[r]^g & z \ar[r]^h & a } & \textnormal{that} & (hg)^f = h^{gf}g^f. \label{16} \\
\textnormal{For} \xymatrix {x \ar[r]^f & y \ar[r]^g & z \ar[r]^h & a } & \textnormal{that} & h^g = (h^{gf})^{g^f}. \label{17} 
\end{eqnarray}

Before we consider these equations again, we notice that from \eqref{11} $j$ is already injective.

\begin{lem}
If j is surjective then $r = t$.
\end{lem}
\begin{proof}
If $j$ is surjective then by \eqref{pull} $q$ is also surjective. Since $rq = tq$ by \eqref{Dom}, we can conclude that $r = t$.
\end{proof}
So with the assumption that $j$ is surjective we see that a skew monoidale in $\Span$ is precisely a category. The extra structure given by $\tau$ and the map $g^f$ reduces to $g^f = g$ for all $f, g \in E$ by \eqref{13}. This recovers the result in $\cite{SkMon}$ where the skew monoidale in $\Span$ assumed the unit was of the form
\begin{align*}
\xymatrix{
& C \ar[dl]_{!} \ar[dr]^1
\\ 1 && C}
\end{align*}

\section{A Characterisation}
\subsection{Coslice Category}
In this subsection we use the notation of \cite{CWM} to denote the coslice category or undercategory of a category, which we now define. 

Let $C$ be a category and $x$ an object of $C$, then the {\em coslice} category denoted by $(x \downarrow C)$ has {\em objects} the arrows of $C$ with source $x$, that is, $\xymatrix{x \ar[r]^f & y}$  which we sometimes denote by the pairs $(f,y)$; and {\em arrows} those $g \colon (f,y) \longrightarrow (f',z)$ where $\xymatrix{y \ar[r]^g & z}$ is an arrow of $C$ such that $f' = gf$, which we usually denote as $\xymatrix{(f,y) \ar[r]^g & (gf,z)}$. It is useful sometimes to write these arrows as the following triangles
\begin{equation*} 
\xymatrix{x \ar[d]_f \ar[dr]^{gf}
 \\ y \ar[r]_g & z}  
\end{equation*}
There is an evident functor $\Cod_{x} \colon (x \downarrow C) \longrightarrow C$ defined on objects by $\xymatrix{x \ar[r]^f & y} \longmapsto y$
and on arrows by $\xymatrix{(f,y) \ar[r]^g
& (gf,z)}$ $\longmapsto$ $\xymatrix{y \ar[r]^g & z}$.

\textbf{Note:} Let $A$ and $B$ be categories and $x$ an object of $A$.
For a functor $T \colon  A \longrightarrow B$ there is an induced functor
$\xymatrix{(x \downarrow  A) \ar[rr]^{(x \downarrow  T)} && (Tx \downarrow  B)}$ 
sending an object  $\xymatrix{x \ar[r]^f & y}$ to  $\xymatrix{Tx \ar[r]^{Tf} & Ty}$
and an arrow 
$$\vbox{ \xymatrix{x \ar[d]_f \ar[dr]^{gf} \\ y \ar[r]_g & z}}
\qquad \longmapsto \qquad
\vbox{ \xymatrix{Tx \ar[d]_{Tf} \ar[dr]^{T(gf)=T(g)T(f)}
    \\ Ty \ar[r]_{Tg} & Tz}} $$

Let $C$ be a category and $f \colon x \longrightarrow y$ be an object of $(x \downarrow C)$; we remind the reader of the coslice category $(f \downarrow (x \downarrow C))$. This category has as its objects the morphisms in $(x \downarrow C)$ starting at $f$ denoted by $\xymatrix{ f \ar[r]^g & gf}$ and as its morphisms the commuting triangles between its objects which we denote by 
\begin{equation*} 
\xymatrix{f \ar[d]_g \ar[dr]^{hg}
\\ gf \ar[r]_h & hgf}
\end{equation*} 
we sometimes denote them by $\xymatrix{g \ar[r]^h & hg}$.

The functor $(f \downarrow \Cod_{x})$ : $\xymatrix{(f \downarrow (x \downarrow C)) \ar[r] & (y \downarrow C)}$ is invertible; it sends an object
$\xymatrix{f \ar[r]^g & gf}$ to $g$ and a morphism
$$\vbox{ \xymatrix{f \ar[d]_g \ar[dr]^{hg}
\\ gf \ar[r]_h & hgf}}
\qquad \longmapsto \qquad
\vbox{ \xymatrix{\Cod(f) \ar[d]_g \ar[dr]^{hg}
\\ \Cod(gf) \ar[r]_h & \Cod(hgf)}} $$
 
\subsection{The Functor $R_x$}
From the previous sections we have seen that a skew monoidale $C$ in $\Span$ gives rise to a category $\mathbb C$ with some extra structure via the function $g^f$ and equations \eqref{11} - \eqref{17}. In this section we use some of these equations to obtain a functor from a coslice category of $\mathbb C$ to $\mathbb C$ and relate the remaining equations to this functor.

For $x \in \mathbb C$ we use equations \eqref{14} and \eqref{16} to {\em define} a functor $R_x \colon (x \downarrow \mathbb C) \longrightarrow \mathbb C$ 
sending an object $\xymatrix{x \ar[r]^f & y}$ to $r(f)$ and an arrow  $\xymatrix{(f,y) \ar[r]^g
& (gf,z)}$ to $\xymatrix{r(f) \ar[r]^{g^f}
 & r(gf)}$.
When it is clear in context we write that on the objects $R_x(f) = r(f)$ and on the arrows $R_x(g) = g^f$.
We check that we do have a functor.
 
We have by definition that $R_x(hg) = (hg)^f$ and $R_x(h)R_x(g) = h^{gf}g^f$ and by \eqref{16} these agree so that $R_x$ preserves composition. Similarly by \eqref{14}, $R_x$ preserves identities and so is a functor.

We now express equation \eqref{17} in terms of the functor $R_x$. However for the benefit of the reader we will explicitly describe the functor 
$(f \downarrow R_x f)$ :$\xymatrix{(f \downarrow (x \downarrow \mathbb C)) \ar[r] & (R_x f \downarrow \mathbb C)}$ which is defined on objects by
$\xymatrix{f \ar[r]^g & gf} \longmapsto \xymatrix{r(f) \ar[r]^{g^f} & r(gf)}$ and on arrows  by 
$$\vbox{ \xymatrix{f \ar[d]_g \ar[dr]^{hg}
\\ gf \ar[r]_h & hgf}}
\qquad \longmapsto \qquad
\vbox{ \xymatrix{r(f) \ar[d]_{g^f} \ar[dr]^{h^{gf} g^f = (hg)^f} \\ r(gf) \ar[r]_{h^{gf}} & r(hgf)}} $$
The above remark allows us to conclude that equation \eqref{17} asserts that the following diagram commutes (it agrees on objects since $r(g^f) = r(g)$).
\begin{equation}\label{factor}
\begin{aligned}
\xymatrix{ (y \downarrow \mathbb C) \ar[rr]^{R_y} && \mathbb C
\\ \\ (f \downarrow (x \downarrow \mathbb C)) \ar[uu]^{(f \downarrow Cod_{x})} \ar[rr]_{(f  \downarrow  R_x f)} && (R_x f \downarrow \mathbb C) \ar[uu]_{R_{R_xf} }}
\end{aligned}
\end{equation}

In the following section we consider the remaining structure involving $U$, $j$, and $\psi$.
\newpage 

\subsection{The Function $E$}

We {\em define} a function $E$ on the set of objects of the category $\mathbb C$ by $E(x) = r(1_x)$. Using \eqref{14} and $r(g^f) = r(g)$ (for a composable pair of morphisms), we note that if $\xymatrix{x \ar[r]^f & y }$ then $E(r(f)) =E(y)$. Taking $f = 1_x$ we find that $E(E(x)) = E(r(1_x)) = E(x)$, so $E$ is idempotent.

From equation \eqref{11}, $\psi j = 1$, and equation \eqref{12}, $j \psi_{x} = r(1_x)$, we can define $U$, $j$, and $\psi$ as a splitting of $E$. So in terms of the functor $R_x$ we have $E(x) = R_x(1_x)$ for each object $x$ in the category $\mathbb C$. With this notation,  equation \eqref{15} then asserts that the following diagram commutes on the objects of the respective categories: 
\begin{equation}\label{ee}
\xymatrix{ Ob(x \downarrow \mathbb C) \ar[rr]^{R_x} \ar[dd]_{Cod_{x}} && Ob(\mathbb C) \ar[dd]^{E}
\\ \\ Ob(\mathbb C) \ar[rr]_{E} && Ob(\mathbb C)}
\end{equation}
Following an object $\xymatrix{x \ar[r]^f & y}$ of $(x \downarrow \mathbb C)$ around \eqref{ee} then asserts in terms of the functor $R_x$ that $R_y(1_y) = R_{R_x f}(1_{R_x f})$ and as $R_x$ is a functor we also have $R_{R_x f}(1_{R_x f}) = R_{R_x f}(R_x(1_y))$.

However if we follow the object $\xymatrix{y \ar[r]^{1_y} & y}$ of $(y \downarrow \mathbb C)$ around \eqref{factor} (really we follow $\xymatrix{f \ar[r]^{1_y} & 1_y f}$ of  $(f \downarrow (x \downarrow \mathbb C))$ around \eqref{factor})  we get that $R_y(1_y) = R_{R_x f}(R_x(1_y))$. So we have shown:
\begin{lem}
If \eqref{factor} holds then so does \eqref{ee}.
\end{lem}  

We now consider the remaining equation \eqref{13} in terms of the functor $R_x$. It is the statement that if for $\xymatrix {j(u) \ar[r]^f & y \ar[r]^g & z }$ then  $g^f = g$.

As $\psi j = 1$ it can be shown that $x = j\psi x$ if and only if there exist a $u$ such that $x = ju$. So for the $u$ where $x = ju$ then $x = E(x) = R_x(1_x)$ (We could now define $U$ to be those $x$ for which $x = R_x(1_x)$). So we conclude that \eqref{13} is the statement that if $x = R_x(1_x)$ then $R_x = \Cod_{x}$.

\textbf{Conclusion:} A skew monoidale $C$ in $\Span$ amounts to a category $\mathbb C$ with

\hspace{2cm}(a) a functor $R_x \colon (x \downarrow \mathbb C) \longrightarrow \mathbb C$ for each $x$ in $\mathbb C$.

\hspace{2cm}(b) if $x = R_x(1_x)$ then $R_x = \Cod_{x}$. 

\hspace{2cm}(c) $R_x$ satisfies \eqref{factor}, that is, for an arrow $\xymatrix{x \ar[r]^f & y}$ in $\mathbb C$ the following commutes
\begin{equation*}
\xymatrix{ (y \downarrow \mathbb C) \ar[rr]^{R_y} && \mathbb C
\\ \\ (f \downarrow (x \downarrow \mathbb C)) \ar[uu]^{(f \downarrow Cod_{x})} \ar[rr]_{(f  \downarrow  R_x f)} && (R_x f \downarrow \mathbb C) \ar[uu]_{R_{R_xf} }}
\end{equation*}

\textbf{Note:} For each $x \in C$, the case when $j = 1$ (equivalently, $j$ is surjective) corresponds to $R_x = \Cod_{x}$.

\subsection{The Simplicial category and the Decalage Functor}

We recall some standard facts about the simplex category $\mathbf{\Delta}$, before using it in our characterisation. There are many references for this section we use \cite{CWM} and \cite{Dus}.

The simplicial category $\mathbf{\Delta}$ has as objects the finite ordinals $\textbf{n} = \{0,1,\dots ,n-1\}  $ and morphisms the order-preserving functions $\xi : \mathbf{m} \longrightarrow \mathbf{n}$ with composition that of functions; the composite of order preserving functions is again order preserving. We note that the ordinal numbers $\textbf{0}$ and $\textbf{1}$ are respectively, initial and terminal objects in $\mathbf{\Delta}$.

If $0\le i\le n$, we write $\delta_i\colon\textbf{n} \to\textbf{n+1}$ for the injective order-preserving function where $\delta_i(k)$ is equal to $k$ if $k<i$ and $k+1$ otherwise (thus its image omits $i$). Similarly, if $0\le i\le n-1$, we write $\sigma_i\colon\textbf{n+1}\to\textbf{n} $ for the order-preserving surjective function where $\sigma_i(k)$ is equal to $k$ if $k\le i$ and $k-1$ otherwise (thus $\sigma_i(i)=\sigma_i(i+1)$, that is, it repeats $i$). We call these maps coface and codegeneracy maps respectively and they satisfy the well known simplicial identities which allow for a presentation of $\mathbf{\Delta}$ with the $\delta_i$ and $\sigma_i$ as its generators and the simplicial identities as its relations.
Moreover, $\mathbf{\Delta}$ has a strict (non-symmetric) monoidal structure $(\mathbf{\Delta}, +,\mathbf{0})$ given by ordinal addition $+ \colon \mathbf{\Delta} \times \mathbf{\Delta} \longrightarrow \mathbf{\Delta}$, defined on ordinals as the ordered sum and on arrows by placing them "side by side". So in terms of the presentation we have that $1_m + \delta_i = \delta_{m+i}$, $1_m + \sigma_i = \sigma_{m+i}$, $\delta_i + 1_m = \delta_i$, and $\sigma_i + 1_m = \sigma_i$ where $1_m$ denotes the identity on $\mathbf{m}$; see \cite{CWM}.
 
A simplicial set is a contravariant functor from $\mathbf{\Delta}$ to \textbf{Set}. The category \textbf{Simp} of simplicial sets and simplicial maps between them is  defined to be the functor category $[\mathbf{\Delta^{op}},\textbf{Set}]$. For a functor $S \colon \mathbf{\Delta^{op}} \longrightarrow \textbf{Set}$ we write $S_n$ for $S(\textbf{n})$. It can be shown that the data for a simplicial set can be specified by the sets $S_n$ and maps $d_i \colon S_n \longrightarrow S_{n-1}$ and $s_i \colon S_n \longrightarrow S_{n-1}$ where for $0\le i\le n$ we define $d_i$ as $S\delta_i$ and $s_i$ as $S\sigma_i$. We call these face and degeneracy maps and they satisfy relations dual to those in $\mathbf{\Delta}$.

For a simplicial set $S$ we consider the shift or (left) decalage functor $\Dec \colon \textbf{Simp} \longrightarrow \textbf{Simp}$ which removes the 0-th face and degeneracy maps, shifts dimension so that $(\Dec(S))_n = S_{n+1}$ and shifts indices on the remaining face and degeneracy maps down by 1 so that $d_i \colon (\Dec(S))_n \longrightarrow (\Dec(S))_{n-1}$ is $d_{i+1} \colon S_{n+1} \longrightarrow S_n$ and $s_i \colon (\Dec(S))_n \longrightarrow (\Dec(S))_{n+1}$ is $s_{i+1} \colon S_{n+1} \longrightarrow S_{n+2}$. Given a simplicial set $S$ as in the diagram
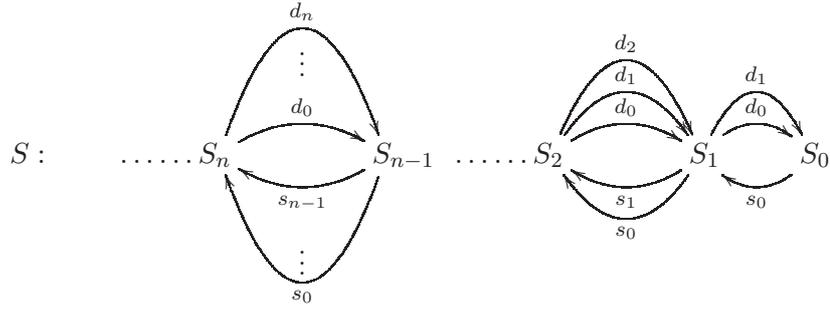
\[ S :\qquad \xymatrix{ 
**[l]\hdots\hdots S_n \ar@/^4pc/[r]^{d_n}_\vdots 
\ar@/^1pc/[r]^{d_0} & **[r]S_{n-1}\ar@/^1pc/[l]^{s_{n-1}}\ar@/^4pc
/[l]^{s_0}_\vdots & **[l] \hdots\hdots S_2 \ar@/^3pc
/[r]^{d_2}\ar@/^2pc
/[r]^{d_1} \ar@/^1pc
/[r]^{d_0}& **[r]S_1 \ar@/^2pc
/[l]^{s_0} \ar@/^1pc
/[l]^{s_1} \ar@/^2pc
/[r]^{d_1}\ar@/^1pc
/[r]^{d_0}& **[r]S_0 \ar@/^1pc
/[l]^{s_0} } \] 
The decalage $\Dec(S)$ of $S$ is the simplicial set
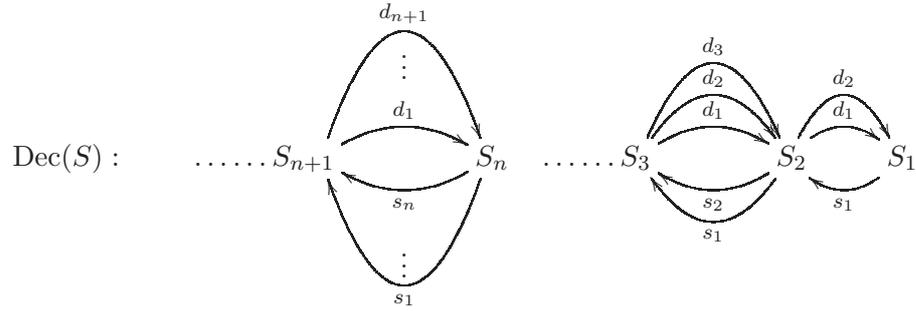
\[ \Dec(S) :\qquad \xymatrix{ 
**[l]\hdots\hdots S_{n+1} \ar@/^4pc/[r]^{d_{n+1}}_\vdots 
\ar@/^1pc/[r]^{d_1} & **[r]S_{n}\ar@/^1pc/[l]^{s_n}\ar@/^4pc
/[l]^{s_1}_\vdots & **[l] \hdots\hdots S_3 \ar@/^3pc
/[r]^{d_3}\ar@/^2pc
/[r]^{d_2} \ar@/^1pc
/[r]^{d_1}& **[r]S_2 \ar@/^2pc
/[l]^{s_1} \ar@/^1pc
/[l]^{s_2} \ar@/^2pc
/[r]^{d_2}\ar@/^1pc
/[r]^{d_1}& **[r]S_1 \ar@/^1pc
/[l]^{s_1} } \] 

There is a simplicial map $\Dec(S) \longrightarrow S$ given (in degree $n$) by the original face map that was discarded $d_0 \colon S_{n+1} \longrightarrow S_n$; we write this map as $d_0 \colon \Dec(S) \longrightarrow S$. The above explicit description of the left decalage construction has a right version and left and right versions for the case of augmented simplicial sets. 

There is a comonad underlying the decalage functor which we briefly describe. Since $\mathbf{1}$ is terminal in $\mathbf{\Delta}$, the arrows $\delta_0 \colon \mathbf{0} \longrightarrow \mathbf{1}$ and $\sigma_0 \colon \mathbf{2} \longrightarrow \mathbf{1}$ form the "universal" monoid $(\mathbf{1},\sigma_0,\delta_0)$ in $\mathbf{\Delta}$. Moreover, there is a monad on $\mathbf{\Delta}$ with endofunctor part $ - + \mathbf{1} \colon \mathbf{\Delta} \longrightarrow \mathbf{\Delta}$, multiplication $- + \sigma_0$, and unit $- + \delta_0$. Now, by reversing the order of each ordinal, $\mathbf{\Delta^{op}}$ contains the universal comonoid $\mathbf{1}$, and as a result we can form a comonad $- + \mathbf{1} \colon \mathbf{\Delta^{op}} \longrightarrow \mathbf{\Delta^{op}}$ dual to the previous construction. This induces by precomposition with the above comonad a comonad on $\mathbf{Simp}$ whose underlying endofunctor is $\Dec \colon \textbf{Simp} \longrightarrow \textbf{Simp}$. The counit of the comonad $\Dec$ is $d_0 \colon S_{n+1} \longrightarrow S_n$. 

Given a category $C$ we can form the nerve $N(C)$ of $C$, it is the well known simplicial set where the face and degeneracy maps are those given in \cite{CWM} and where 
 
\hspace{2.5cm}$N(C)_0$ = set of objects in $C$

\hspace{2.5cm}$N(C)_1$ = set of morphisms in $C$

\hspace{2.5cm}$N(C)_2$ = set of composable pairs of morphisms in $C$

\hspace{4cm}\vdots

\hspace{2,5cm}$N(C)_n$ = set of composable $n$-tuples of morphisms in $C$.

With the above discussion in mind we see that if $C$ is a category then so is $\Dec(C)$  where

\hspace{2.5cm}$\Dec(C)_0$ = set of morphisms in $C$

\hspace{2.5cm}$\Dec(C)_1$ = set of composable pairs of morphisms in $C$

\hspace{4.2cm}\vdots

\hspace{2.5cm}$\Dec(C)_{n}$ = set of composable $(n+1)$-tuples of morphisms in $C$.

Recall that in the category \textbf{Cat} of small categories and functors, the coproduct of a family of categories is their disjoint union. For $I$ a set and $(C_{i})_{i \in I}$ a family of objects in \textbf{Cat} we write $\coprod_{i \in I} C_{i}$ for the coproduct of the family $(C_{i})_{i \in I}$. Now with this notation and from the functors $\Cod_{x}$ we can form a functor from  $\coprod_{x \in C} (x \downarrow C)$ to $C$, which we denote by $\Cod$.

Having described above what the functor $\Dec$ does on objects of \textbf{Cat} we notice for a category $C$, that $\Dec(C) = \coprod_{x \in C} (x \downarrow C)$. So to complete this (brief) description of $\Dec$ as an endofunctor from \textbf{Cat} we need to describe what it does on arrows of \textbf{Cat}.

Let $F \colon X \longrightarrow C$ be a functor where $X$ and $C$ are categories. As we need a functor from a coproduct in \textbf{Cat}, it is sufficient, for each $x \in X$, to specify a functor from $(x \downarrow X)$ to $\Dec(C)$ where $\Dec(C) = \coprod_{c \in C} (c \downarrow C)$. We define the functor $\Dec(F)_{x} \colon (x\downarrow X) \longrightarrow \Dec(C)$ by the following composite
\begin{align*}
\xymatrix{(x \downarrow X) \ar[rr]^{(x \downarrow F)}
 && (F(x)\downarrow C) \ar[rr]^{inclusion} && \Dec(C)}
\end{align*}
Thus we have a functor $\Dec(F) \colon \Dec(X) \longrightarrow \Dec(C)$. 

We complete this description of $\Dec$ as an endofunctor on \textbf{Cat} with the observation that $\Cod \colon \Dec(C) \longrightarrow C$ is the map $d_0 \colon \Dec(C) \longrightarrow C$ and note that any coslice category can be extracted from $\Dec(C)$ using this $d_0$.

Using these constructions we can rewrite the previous description of a skew monoidale in $\Span$ as:

\textbf{Conclusion:} A skew monoidale $C$ in $\Span$ amounts to a category $\mathbb C$ with

\hspace{2cm}(a) a functor $R \colon \Dec(\mathbb C) \longrightarrow \mathbb C$, where

\hspace{2cm}(b) $R$ makes the following diagram  commute
 
\begin{equation*}
\begin{aligned}
\xymatrix{ \Dec(\mathbb C) \ar[rr]^R && \mathbb C
\\ \\ \Dec(\Dec(\mathbb C)) \ar[rr]_{\Dec(R)} \ar[uu]^{\Dec(\Cod)} && \Dec(\mathbb C) \ar[uu]_{R}}
\end{aligned}
\end{equation*}

\hspace{2cm}(c) such that, if $x = R(1_x)$ then $R_x = \Cod_x$. 

Note that when starting with just a category then $R = Cod$.

\section{An Example}
In this section we denote by $(M,\mu,\eta)$ or just $M$ a monoid in the monoidal category $(\textbf{Set},\times,1)$ where the tensor product is the cartesian product $\times$ and $1=\{\star \}$ denotes a one point set as its unit. Here the two arrows $\mu$ and $\eta$ in $\textbf{Set}$ satisfy the usual equations (see \cite{CWM}). For $\mu \colon M \times M \longrightarrow M$ and for $a,b \in M$ we write $\mu(a,b)= a.b$ and write for $\eta(\{\star \}) = 1_M$, we sometimes just write $\eta(\{\star \}) = 1$ where it should be clear in context what $1$ represents.

We recall the embedding $(-)_{\star}
\colon \textbf{Set} \longrightarrow \Span$ which is the identity on objects and assigns to the morphism $f \colon A \longrightarrow B$ the following span 
\begin{align*}
f_{\star} = \xymatrix{& A \ar[dl]_{1_{A}} \ar[dr]^f
\\ A && B}
\end{align*}
In fact, this is a strong monoidal pseudofunctor and as a consequence sends monoids in \textbf{Set} to monoidales in $\Span$.
We can therefore consider a monoid $(M,\mu,\eta)$ in \textbf{Set} as a (skew) monoidale in $\Span$.  

The 1-cell  $\xymatrix{ p \colon  C\times C \ar[r]|(.65){\vert} & C}$ for a skew monoidale in $\Span$ is given by
\begin{align*}
\xymatrix{ & M\times M \ar[dl]_{(\pi_{1},\pi_{2})} \ar[dr]^{\mu}
\\ M\times M && M}
\end{align*}
where $\pi_{i} \colon M\times M \longrightarrow M$ is defined by $\pi_{i}(m_1,m_2) = m_i$ for $i$=1,2 and $m_1,m_2 \in M$.

The 1-cell $\xymatrix{ j \colon  1 \ar[r]|(.53){\vert} & C}$ for a skew monoidale in $\Span$ is given by
\begin{align*}
\xymatrix{ & 1 \ar[dl]_1 \ar[dr]^{\eta}
\\ 1 && M}
\end{align*}
With these choices for $p$ and $j$, the 2-cell $\rho \colon  1 \Longrightarrow p(1 \times j)$ for this skew monoidale is given by the following diagram 
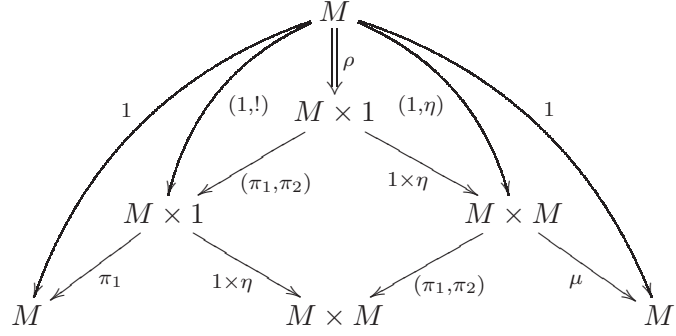
\begin{align*}
\xymatrix{ && M \ar@{=>}[d]^{\rho} \ar@/_1.8pc/[dddll]_1 \ar@/^1.2pc/[ddr]_{(1,\eta)} \ar@/_1.2pc/[ddl]^{(1,!)}
\ar@/^1.8pc/[dddrr]^1
\\ && M\times 1 \ar[dr]_{1 \times \eta} \ar[dl]^{(\pi_{1},\pi_{2})}
\\ & M \times 1 \ar[dr]_{1 \times \eta} \ar[dl]^{\pi_1} && M\times M \ar[dr]_{\mu} \ar[dl]^{(\pi_{1},\pi_{2})}
\\ M  && M\times M  && M}
\end{align*}

and the 2-cell $\alpha \colon  p(p \times 1) \Longrightarrow p(1 \times p)$ is given by
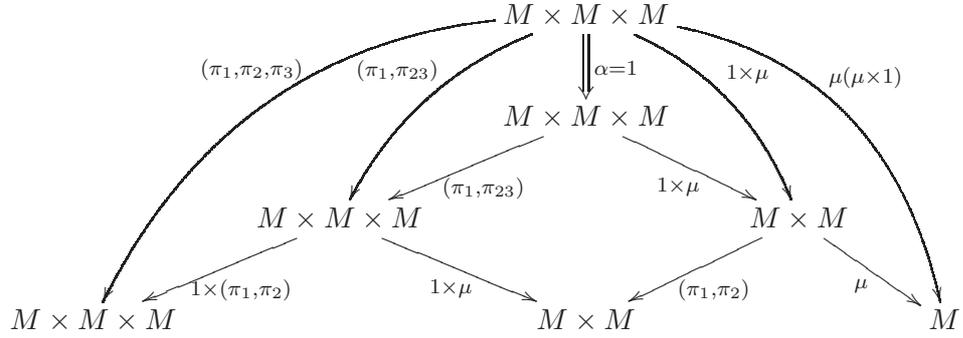
\begin{align*}
\xymatrix{ && M\times M\times M \ar@{=>}[d]^{\alpha=1} \ar@/_3pc/[dddll]_{(\pi_{1},\pi_{2},\pi_{3})} \ar@/^1.2pc/[ddr]^{1\times \mu} \ar@/_1.2pc/[ddl]_{(\pi_{1},\pi_{23})}
\ar@/^3pc/[dddrr]^{\mu(\mu \times 1)}
\\ && M\times M\times M \ar[dr]_{1\times \mu} \ar[dl]^{(\pi_{1},\pi_{23})}
\\ & M \times M\times M \ar[dr]_{1\times \mu} \ar[dl]^{1\times (\pi_{1},\pi_{2})} && M\times M \ar[dr]_{\mu} \ar[dl]^{(\pi_{1},\pi_{2})}
\\ M\times M\times M  && M\times M  && M}
\end{align*}
where $\pi_{23} \colon M\times M\times M \longrightarrow M\times M$ is defined as $\pi_{23}(m_{1},m_{2},m_{3}) = (m_{2},m_{3})$. 
We will now describe the resulting monoidale in terms of the characterisation of skew monoidales in $\Span$ given in the previous sections.

So with these choices for $p$ and $j$, $M$ is a category whose {\em objects} are the elements of the set $M$ and whose {\em arrows} are the pairs $(a,b) \in M\times M$ with source $\pi_{1}(a,b) = a$ and target $\mu(a,b) = a.b$ which we represent as $\xymatrix{a \ar[r]^{b} & a.b}$.
The composition of arrows in $M$ and the functor $R \colon \Dec(M) \longrightarrow M$ are both defined by the 2-cell  $\alpha \colon  p(p \times 1) \Longrightarrow p(1 \times p)$. The composition of arrows in $M$ is then given by
$\xymatrix{(a,b,c) \ar@{|->}[rr]^{1\times \mu} && (a,b.c)}$ for $(a,b,c) \in M\times M\times M$ and so the composite
$\xymatrix{a \ar[r]^b & a.b \ar[r]^(.4)c & (a.b).c}$ is given by  $\xymatrix{ a \ar[r]^(.3){b.c} & a.(b.c) }$.

For the functor $R \colon \Dec(M) \longrightarrow M$ and the $p$ and $j$ chosen from $M$ we have on the objects of $\Dec(M)$ that $R((a,b)) = \pi_{2}(a,b) = b$ or $R(\xymatrix{a \ar[r]^b & a.b}) = b$ and on the arrows of $\Dec(M)$ we have that $R((a,b,c)) = \pi_{23}(a,b,c) = (b,c)$ or $R(\xymatrix{a \ar[r]^b & a.b \ar[r]^(.4)c & (a.b).c}) = \xymatrix{b \ar[r]^(.4)c & b.c}$.

The identity arrow for the category $M$ exists via the 2-cell  $\rho \colon  1 \Longrightarrow p(1 \times j)$ and is represented as $\xymatrix{a \ar[r]^(0.3){1} & a.1=a}$.

\begin{remark}
The monoids in $\textbf{Set}$ constitute a category $\textbf{Mon}$ and the above example defines the object part of a functor $T \colon \textbf{Mon} \longrightarrow \textbf{Cat}$. For a morphism of monoids $f \colon (M,\mu,\eta) \longrightarrow (M',\mu',\eta')$ the induced functor $TM \longrightarrow TM^{'}$ sends an object $m$ to $fm$ and a morphism $(m,n)$ to $(fm,fn)$. 
\end{remark}

\begin{remark}
Considering a category as a partial monoid and using the notation of \cite{CWM} we can generalize the above example; we can instead start with a (small) category $C$ where $O$, $A$ and $A \times_{O} A$ respectively denotes the sets of objects, arrows and composable arrows of $C$.  

The tensor for a monoidale in $\Span$ is given by
\begin{align*}
\xymatrix{ & A\times_{O} A \ar[dl]_{(\pi_{1},\pi_{2})} \ar[dr]^{comp}
\\ A\times A && A}
\end{align*}

The unit for that monoidale in $\Span$ is given by
\begin{align*}
\xymatrix{ & O \ar[dl]_{!} \ar[dr]^{id}
\\ 1 && A}
\end{align*} 
\end{remark}

\begin{remark}
The following observations are due to Joachim Kock who has noted that for the example starting with a monoid $M$ the above construction of a (skew) monoidale is just the category $\Dec(M)$ and the functor $T \colon \textbf{Mon} \longrightarrow \textbf{Cat}$ is the restriction of the functor $\Dec \colon \textbf{Cat} \longrightarrow \textbf{Cat}$. Similarly, the example starting with a category $C$, the corresponding construction of a (skew) monoidale is $\Dec(C)$. Now starting with a category $C$ we consider the category $D = \Dec(C)$ so now using our previous notation $D_0 = C_1$, $D_1 = C_2$ and so on. The extra structure required on this category $D$ to be considered as a skew monoidale in $\Span$ is a functor $R \colon \Dec(D) \longrightarrow D$ which amongst other conditions agrees with the counit $d_0$ (that is with $\Cod \colon \Dec(D) \longrightarrow D$) on the objects $U$ in $D$ but now since $D = \Dec(C)$ this $d_0$ from $D$ is $d_1$ from $C$. Now a natural choice for the functor $R$ would be $\Dec(d_0)$ where $d_0$ is from $C$ and to define $U$, (which are the objects $x$ in $D$ for which $x = R_x(1_x)$ and since $C_1=D_0$) we can use $\Dec(s_0)$ where $s_0 \colon C_0 \longrightarrow C_1$ is from $C$, previously omitted by $\Dec$. Hence, for the above examples, the extra face map $R$ is available naturally resulting in the monoidale being non-skew.
\end{remark}
\begin{remark}
The following is a non-trivial example given by Stephen Lack at a talk to the Australian Category Seminar \cite{Ltalk}.
Batanin and Markl in \cite{Bat} define a strict operadic category as a category $\mathbb C$ equipped with a cardinality functor into \textbf{sFSet}, the skeletal category of finite sets, where each connected component of $\mathbb C$ has a chosen terminal object. One of the axioms for a strict operadic category requires the existence of a family of functors from a slice category of $\mathbb C$ into $\mathbb C$, for the chosen terminal object this is required to be the domain functor. Lack has shown that strict operadic categories are equivalent to left normal skew monoidales in $\Span([\mathbb{N},\textbf{Set}])$. 
Here $\mathbb{N}$ denotes the set of natural numbers, seen as a discrete category, and the functor category $[\mathbb{N},\textbf{Set}]$ is given a monoidal structure via Day's convolution. 
\end{remark}

\subsubsection*{\textbf{Acknowledgements:}}
I wish to thank my supervisor, Stephen Lack, whose patience and guidance made possible the writing of this article and Joachim Kock for his comments.

\bibliographystyle{amsplain}

\end{document}